\documentclass[11pt, psamsfonts]{amsart}
\usepackage{amsmath}
\usepackage{amsthm}
\usepackage{amssymb}
\usepackage{amscd}
\usepackage{amsfonts}
\usepackage{amsbsy}
\usepackage{epsfig}

\newtheorem{theorem}{Theorem}
\newtheorem{remark}{Remark}
\newtheorem{proposition}{Proposition}
\newtheorem{lemma}{Lemma}
\newtheorem{definition}{Definition}

\newtheorem{claim}{Claim}
\newtheorem{example}{Example}
\newtheorem*{claim*}{Claim}

\def\ie{{\em i.e.,\ }}

\newfont\bbf{msbm10 at 12pt}

\def\R{{\mathbb R}}

\def\N{{\mathbb N}}

\def\B{{\mathcal B}}

\def\D{{\mathcal D}}
\def\M{{\mathcal M}}
\def\A{{\mathcal A}}
\def\B{{\mathcal B}}

\def\Cr{\mbox{Cr}\,}
\def\M{\mbox{M}\,}
\def\Ex{\mbox{Ex}\,}

\def\diam{\mbox{\rm diam} }
\def\card{\mbox{\rm card} }
\def\deg{\mbox{\rm deg} }

\def\le{\leqslant}
\def\ge{\geqslant}
\def\htop{h_{top}}

\def\1{ {\hbox{{\it 1}} \!\! I} }

\begin{document}

\title[Entropy of Banach spaces]
{The topological entropy of Banach spaces}
\author{Jozef Bobok and Henk Bruin}
\date{\today}
\thanks{The first author was partly supported by the Grant Agency of the Czech Republic contract number
201/09/0854.
The second author would like to thank the University of
Surrey Faculty Research Support Fund (FRFS). Both
authors gratefully acknowledge the support of the MYES of the
Czech Republic via contract MSM 6840770010.}
\address{KM FSv \v CVUT, Th{\'a}kurova 7,
166 29 Praha 6,
Czech Republic }
\email{bobok@mat.fsv.cvut.cz}
\urladdr{http://mat.fsv.cvut.cz/bobok/}
\address{Department of Mathematics,
University of Surrey,
Guildford GU2 7XH, UK }
\email{h.bruin@surrey.ac.uk}
\urladdr{http://personal.maths.surrey.ac.uk/st/H.Bruin/}

\begin{abstract}
We investigate some properties of (universal) Banach spaces of real functions
in the context of topological entropy.
Among other things, we show that any subspace of $C([0,1])$ which is isometrically isomorphic to $\ell_1$ contains a functions with infinite
topological entropy. Also, for any $t \in [0, \infty]$, we construct a
(one-dimensional) Banach space in which any nonzero function has
topological entropy equal to $t$.
\end{abstract}

\maketitle

\section{Introduction}\label{sec:intro}

Let $C([0,1])$ denote the set of all continuous functions
$f: [0,1] \to \R$ equipped with the supremum norm. A theorem of Banach and Mazur \cite{Ba32} states that the Banach space $C([0,1])$ is universal, \ie every real, separable Banach space $X$ is isometrically isomorphic to a closed subspace of $C([0, 1])$. It is known that one can require more properties of the functions of $C([0,1])$ in the image of $X$:  a universal space containing only the zero function and nowhere differentiable functions \cite{Ro95}, resp.\ consisting of the zero function and nowhere approximatively differentiable and nowhere H\"{o}lder functions \cite{He00} has been proved. On the other hand, no universal space can consist of functions of bounded variation \cite{LM40} and every
isometrically isomorphic copy of $\ell_1$
(\ie the space of sequences with $1$-norm) in $C([0,1])$
contains a function which is non-differentiable at every point of a perfect subset of $[0,1]$, see \cite{PT84}.

In this paper we are going to investigate some properties of (universal) Banach spaces of real functions of a real variable in the context of topological entropy. We show how to construct a universal Banach space using the zero function and functions with infinite topological entropy - Theorem~\ref{thm:A}, and as a supplement of the result from \cite{PT84} we show that any subspace of $C([0,1])$ which is isometrically isomorphic to $\ell_1$ contains a functions with infinite topological entropy - Theorem~\ref{thm:E}. Finally, for any $t \in [0, \infty]$ we construct a (one-dimensional) Banach space in which any nonzero function has its topological entropy equal to $t$  - Theorem~\ref{thm:D}.

\section{Preliminaries and auxiliary results}

Let $C_b(X)$ denote the set of all {\it bounded} continuous functions
$f: X \to \R$ equipped with the supremum norm. Clearly, $C_b(\R)$ is a non-separable Banach space. Let $[a,b]$ be a closed finite subinterval of $\R$. We identify $f\colon~[a,b]\to\R$ with its extension
\begin{equation}\label{e:1}
(\Ex f)(x) = \left\{ \begin{array}{ll}
f(x) & \text{ if } x \in [a,b]; \\[1mm]
f(b) & \text{ if } x \geq b; \\[1mm]
f(a)  & \text{ if } x \leq a.
\end{array} \right.
\end{equation}
Under this identification, $C([a,b])\subset C_b(\R)$. We will deal with the topological entropy  of maps from $C_b(\R)$ defined as $\htop(f):=\htop(f\vert_{ \overline{f(\R)}})$ - see \cite[Chapter 4]{alm00}.

The well known Banach - Mazur Theorem states that the Banach space $C([0,1])$ is universal, \ie every real, separable Banach space $\D$ is isometrically
isomorphic to a closed subspace of $C([0, 1])$. Since by our convention, $C([0,1])$ is a closed subspace of $C_b(\R)$, the non-separable space $C_b(\R)$ is also universal. In our paper we will restrict ourselves to separable universal Banach spaces only.

Following \cite{alm00} we recall the notion of horseshoe.

\begin{definition}\label{def:A}
A function $f\in C_b(\R)$ is said to have a {\em $d$-horseshoe}
if there exist $d$ subintervals $I_1, I_2, \dots, I_d$ of $\R$ with disjoint interiors such that $f(I_i) \supset I_j$ for all $1 \leq i,j \leq d$.
\end{definition}

\begin{proposition}\label{prop:A}\cite{MS}~If $f\in C_b(\R)$ has a {\em $d$-horseshoe} then $\htop(f) \geq \log d$.\end{proposition}

In the next lemma we denote by $F(X)$ a linear space of functions $f\colon~X\to\R$.

\begin{lemma}\label{lem:indep}
Given $n$ linearly independent functions in $F(X)$,
there exist $n$ points $x_1, \dots, x_n \in X$ such that the vectors
$$
\left( \begin{array}{c}
f_1(x_1) \\ f_1(x_2) \\ \vdots \\ f_1(x_n)
\end{array} \right),
\left( \begin{array}{c}
f_2(x_1) \\ f_2(x_2) \\ \vdots \\ f_2(x_n)
\end{array} \right),  \dots ,
\left( \begin{array}{c}
f_n(x_1) \\ f_n(x_2) \\ \vdots \\ f_n(x_n)
\end{array} \right)
$$
are linearly independent in $\R^n$.
\end{lemma}

\begin{proof}
This is clear if $n = 1$.
Assume now by induction that the lemma holds for $k < n$ and points $x_1, \dots, x_k$.
For the unique linear combination such that $f_{k+1}(x_i) = a_1f_1(x_i) + \dots + a_k f_k(x_i)$
for all $1 \leq i \leq k$.
Now if $f_{k+1}(x) = a_1f_1(x) + \dots + a_k f_k(x)$ for all
$x \in X$, then $f_1, \dots, f_{k+1}$ are linearly dependent, contrary to our assumption.
So there must be some other point $x_{k+1}$ for which
$f_{k+1}(x_{k+1}) \neq a_1f_1(x_{k+1}) + \dots + a_k f_k(x_{k+1})$, which concludes the
induction step.
\end{proof}

\section{The Main Theorems}\label{sec:theorems}

\begin{definition}\label{def:B}
For a given set $\B \subset C_b(\R)$,
let
\begin{eqnarray*}
\htop^+(\B) &=& \sup\{ \htop(f) : f \in \B\}, \\[2mm]
\htop^-(\B) &=& \inf\{ \htop(f) : f \in \B,~f\text{ is non-zero}\}.
\end{eqnarray*}
\end{definition}

\begin{theorem}\label{thm:B}
If a linear space $\B \subset C_b(\R)$ has dimension $n$, then $$\htop^+(\B) \geq \log(n-1).$$
 In particular, $\htop^+(\B) = \infty$ if $\dim(\B) = \infty$.
\end{theorem}

\begin{proof}
Take $f_1, \dots, f_n \in \B$ linearly independent and find points $x_1, \dots, x_n$
as in Lemma~\ref{lem:indep}. We can assume that $x_1 < x_2 < \dots < x_n$.
Form a linear combination
$f = a_1f_1 + \dots + a_n f_n$ such that $f(x_i) = x_1$ if $i$ is odd, and
 $f(x_i) = x_n$ if $i$ is even.
 Then $f$ has an $(n-1)$-horseshoe, so from Proposition~\ref{prop:A} we get $\htop(f) \geq \log(n-1)$ as required.
\end{proof}

\begin{example}
(i)~Let $[a,b]$ be a closed subinterval of $\R$. Given a continuous function $f:\R \to \R$, let
$$
(\Cr_{[a,b]} f)(x) = \left\{ \begin{array}{ll}
f(x) & \text{ if } x \in [a,b]; \\[1mm]
f(b) & \text{ if } x \geq b; \\[1mm]
f(a)  & \text{ if } x \leq a,
\end{array} \right.
$$
be the {\em cropped} version of $f$. Clearly $\Cr_{[a,b]} f \in C([a,b])\subset C_b(\R)$.
Let
$$
P^{n-1} = \{ \Cr_{[a,b]} p : p\in C(\R)\text{ is a polynomial of degree } \leq n-1\}.
$$
Then $P^{n-1}$ has dimension $n$, each $f \in P^{n-1}$ is at most $n-2$-modal, so our definition of the entropy $\htop (f)$ and \cite[Theorem 4.2.4]{alm00} imply that $\htop (f)\leq \log (n-1)$.
This shows that the bound in Theorem~\ref{thm:B} is sharp.

(ii) Let
$$
P = \{ \Cr_{[a,b]} p : p\in C(\R)\text{ is a polynomial }\}.
$$
Then $P$ is a normed linear subspace of $C([a,b])$ and $\dim(P)=\infty$, hence by Theorem~\ref{thm:B}, $\htop^+(P) = \infty$. By the same argument as above, the entropy  of any $p\in P$ satisfies $\htop(p)\le\log\deg(p)$, so it is finite.
\end{example}

Theorem~\ref{thm:B} does not answer the question whether every infinite dimensional Banach space $\A\subset C_b(\R)$ contains a function with infinite entropy. Our next example shows that in general it is not the case.

\begin{example}
For $n \ge 1$ and $a \in \R$, let $f_{n, a}:\R \to \R$ be given by
$$
f_{n, a}(x) = \left\{ \begin{array}{ll}
a \cdot (x-2+\frac1n) \cdot (2-\frac{1}{n+1}-x) & \text{ if } x \in J_n :=  [2-\frac{1}{n}, 2-\frac{1}{n+1}]; \\[1mm]
0  & \text{ otherwise.}
\end{array} \right.
$$
Clearly $f_{n,a}$ is unimodal, so its entropy $\htop(f_{n,a}) \leq \log 2$.
Consider the smallest Banach space $Q$ 
(subspace of $C_b(\R)$ with supremum norm) 
containing all finite sums $f_{1,a_1}+f_{2,a_2}+\cdots+f_{n, a_n}$.
Then $dim(Q)=\infty$, $\lim_{x\to 2_{-}}f(x)=0$ for each $f\in Q$ and if 
$$
\max\{x\in\R\colon~f(x)=1\}\in J_n, 
$$
then since the modality of $f|_{[1,1-\frac{1}{n+1}]}$ is at most $2n$
and $f^2(x) = 0$ for $x \notin [1,1-\frac{1}{n+1}]$, we conclude that 
$\htop(f)\le\log(2n+1)$. 
\end{example}

As a counterpart of the previous example we will prove the following theorem.

\begin{theorem}\label{thm:E}
Let $\A\subset C([0,1])$ be isometrically isomorphic to $\ell_1$. Then $\A$ contains a function with infinite topological entropy.\end{theorem}
\begin{proof}Let $\Phi$ be an isometrical isomorphism ensured by the statement, so $\Phi(\ell_1)=\A$.
For $i\in\N$ let $e_i=(e_{ij})_{j=1}^{\infty}\in\ell_1$ be defined by
\begin{equation*}
e_{ij}=\delta_{ij},
\end{equation*}
where $\delta_{ij}$ is the Kronecker delta. Then for every $n\in\N$ and every choice of distinct positive integers $i(1), \dots,i(n)$
\begin{equation*}\| \pm e_{i(1)}\pm \cdots\pm e_{i(n)}\|_{\ell_1}=n.\end{equation*} Denote $f_i=\Phi(e_i)\in\A\subset C([0,1])$, $i\in\N$. Clearly $\|f_i\|=1$; in particular for every $x\in [0,1]$,\begin{equation}\label{e:7}\vert f_i(x)\vert\le 1.\end{equation}
\begin{claim}\label{cl:2}For every $s=(s_i)_i\in\{1,-1\}^{\N}$ there exists a point $x\in [0,1]$ such that the sequence $(f_i(x))_{i\in\N}$ is equal to either
$s$ or $-s$.
\end{claim}
\begin{proof}
Assume that for some $n$,
\begin{equation*}
\forall~x\in [0,1]\colon~(f_i(x))_{i=1}^n\neq (s_i)_{i=1}^n\text{ and }(f_i(x))_{i=1}^n\neq (-s_i)_{i=1}^n;
\end{equation*}
then \eqref{e:7} implies $\vert \sum_{i=1}^ns_if_i(x)\vert<n$ for every $x\in [0,1]$. This contradicts the equalities
\begin{equation}
\|\sum_{i=1}^ns_ie_i\|_{\ell_1}=\|\sum_{i=1}^ns_if_i\|=n.
\end{equation}
Thus, for each $n\in\N$ one can find a point $x_n\in [0,1]$ for which either $(f_i(x_n))_{i=1}^n=(s_i)_{i=1}^n$ or $(f_i(x_n))_{i=1}^n=(-s_i)_{i=1}^n$. Taking a limit point $x$ of the sequence $(x_n)_n$, from the continuity of the functions $f_i$ we get either $(f_i(x))=s$ or $(f_i(x))=-s$.
\end{proof}

For $n>1$ define the matrix $A_n=(a^n_{ij})_{i,j=1}^n$ by $a_{ij}=(-1)^i$ for $1\le j<i$ and $a_{ij}=(-1)^{i+1}$ when $i\le j\le n$.
For instance, the particular matrix $A_8$ is
$$
\begin{pmatrix}
+1 &\  +1 &\  +1 &\  +1 &\  +1 &\  +1 &\  +1 &\  +1 \\
+1 &\  -1 &\  -1 &\  -1 &\  -1 &\  -1 &\  -1 &\  -1 \\
-1 &\  -1 &\  +1 &\  +1 &\  +1 &\  +1 &\  +1 &\  +1 \\
+1 &\  +1 &\  +1 &\  -1 &\  -1 &\  -1 &\  -1 &\  -1 \\
-1 &\  -1 &\  -1 &\  -1 &\  +1 &\  +1 &\  +1 &\  +1 \\
+1 &\  +1 &\  +1 &\  +1 &\  +1 &\  -1 &\  -1 &\  -1 \\
-1 &\  -1 &\  -1 &\  -1 &\  -1 &\  -1 &\  +1 &\  +1 \\
+1 &\  +1 &\  +1 &\  +1 &\  +1 &\  +1 &\  +1 &\  -1 \\
\end{pmatrix}
$$

One can easily verify the following fact.

\begin{claim}\label{cl:3}For any ${\bf \beta}\in\R^n$, the linear equation $A_n{\bf \alpha}={\bf \beta}$ has a unique solution ${\bf \alpha}$ given by the formulas
\begin{equation}\label{e:10}
\alpha_i=\frac{\beta_i+\beta_{i+1}}{(-1)^{i+1}2},~i=1,\dots,n-1,
\quad \alpha_n=\frac{\beta_1+(-1)^{n+1}\beta_n}{2}.
\end{equation} In particular, $\max\vert \alpha_i\vert\le \max\vert \beta_i\vert$.\end{claim}

Let us denote the $i$-th row of the matrix $A_n$ by $a^n_i=(a^n_{i1},a^n_{i2},\cdots,a^n_{in})$. By Claim~\ref{cl:2}, for each $n>1$ there are distinct
points $x^n_1,\dots,x^n_n\in [0,1]$ such that either
\begin{equation}\label{e:8}
f_1(x^n_i)=\cdots =f_{2^n}(x^n_i)=1,~(f_{2^n+1}(x^n_i),f_{2^n+2}(x^n_i),\cdots,f_{2^n+n}(x^n_i))=a^n_i,
\end{equation}
or
\begin{equation}\label{e:9}
f_1(x^n_i)=\cdots =f_{2^n}(x^n_i)=-1,~(f_{2^n+1}(x^n_i),f_{2^n+2}(x^n_i),\cdots,f_{2^n+n}(x^n_i))=-a^n_i.
\end{equation}

Put $X_n=\{x^n_1,\dots,x^n_n\}$. Since $n=\card(X_n)$ is growing to infinity,
one can consider subsets $X'_n\subset X_n$ satisfying
\begin{equation}\label{e:11}
\lim_{n\to\infty}\card(X'_n)=\infty,~\lim_{n\to\infty}\diam(X'_n)=0.
\end{equation}
Passing to a subsequence if necessary, we can assume that $X'_n\rightarrow x_0\in [0,1]$, \ie
\begin{equation}\label{e:12}\forall~\varepsilon>0~\exists~n_0~\forall~n>n_0\colon~X'_n\subset (x_0-\varepsilon,x_0+\varepsilon).\end{equation}

Now, using \eqref{e:8} and \eqref{e:9}, we obtain that either $(f_i(x_0))_i=(1)_i$ or $(f_i(x_0))_i=(-1)_i$. Without loss of generality assume the first possibility. Notice that then
\begin{equation}\label{e:13}
\forall~n>1\colon~(f_{2^n+1}(x_0),f_{2^n+2}(x_0),\cdots,f_{2^n+n}(x_0))=a^n_1.
\end{equation}

We can formally put

\begin{equation}\label{e:14}
e=x_0e_1+\sum_{n=2}^{\infty}\sum_{k=1}^n\alpha^n_ke_{2^n+k},~\Phi(e)=f(x)=x_0f_1(x)+\sum_{n=2}^{\infty}\sum_{k=1}^n\alpha^n_kf_{2^n+k}(x),
\end{equation}

where coefficients $\alpha^n=(\alpha^n_1,\alpha^n_2,\cdots,\alpha^n_n)$ satisfy a linear equation $A_n\alpha^n=\beta^n$, $\beta^n=(\beta^n_1,\beta^n_2,\cdots,\beta^n_n)\in\R^n$. It can be easily seen that $f\in \A$
if and only if
\begin{equation*}
\sum_{n=2}^{\infty}\sum_{k=1}^n\vert\alpha^n_k\vert<\infty.
\end{equation*}

Moreover, if $\beta^n_1=0$ for each $n$ and $f\in C([0,1])$ then $f(x_0)=x_0$
by the equation $x_0f_1(x_0)=x_0$ and the property \eqref{e:13} implying
\begin{equation*}
\sum_{k=1}^n\alpha^n_kf_{2^n+k}(x_0)=0\text{ for each }n.
\end{equation*}

Using Claim~\ref{cl:3} we will show in the sequel that there exists a sequence $(\beta^n=(0,\beta^n_2,\cdots,\beta^n_n))_n$ such that the corresponding function $f$ given by \eqref{e:14} satisfies $f\in \A$ and $\htop(f)=\infty$. In what follows we denote
\begin{equation*}
g_1(x)=x_0f_1(x),~g_m(x)=g_1(x)+\sum_{n=2}^m\sum_{k=1}^n\alpha^n_kf_{2^n+k}(x),~m\ge 2.
\end{equation*}
Let $\omega(f,X) = \sup_{x,y\in X}\vert f(x)-f(y)\vert$
denote the oscillation of a function $f$ on a set $X$.
For a positive $\varepsilon(i)$ we use the notation $J(i)=[x_0-\varepsilon(i),x_0+\varepsilon(i)]$. The zero element in $\R^n$ is denoted by $0_n$. Let $(\gamma_m)_{m\in\N}$ be a sequence of positive numbers satisfying for each $m$
\begin{equation}\label{e:17}
\gamma_m>\sum_{i=m+1}^{\infty}2(i+3)\gamma_i.
\end{equation}
\vskip1mm
{\bf Step 0.} $n(0)=1$.
\vskip1mm
{\bf Step 1.} We can find values $\varepsilon(1)>0$ and $n(1)>n(0)+1$ such that
\begin{equation}\label{e:16}
\varepsilon(1)+\omega(g_{n(0)},J(1))<\gamma_1-\sum_{i=2}^{\infty}2(i+3)\gamma_i,
\end{equation}
\begin{equation*}
J(1)\cap X'_{n(1)}\supset \{x^{n(1)}_{i(1)}<x^{n(1)}_{i(2)}<x^{n(1)}_{i(3)}<x^{n(1)}_{i(4)}\}.
\end{equation*}
We put $\beta^n=0_n$ for each $n(0)<n<n(1)$; the coefficients $\alpha^{n(1)}_k$, $k=1,\dots,n(1)$ are gained as the unique solution of the linear equation $A_{n(1)}\alpha^{n(1)}=\beta^{n(1)}$, where (as we already know) $\beta^{n(1)}_1=0$, $\beta^{n(1)}_{i(j)}=(-1)^j\gamma_1$, $j=1,2,3,4$ and $\beta^{n(1)}_i=0$ otherwise.

\vskip1mm
{\bf Step m.} We can find values $\varepsilon(m)>0$ and $n(m)>n(m-1)+1$ such that
\begin{equation}\label{e:15}
\varepsilon(m)+\omega(g_{n(m-1)},J(m))<
\gamma_m-\sum_{i=m+1}^{\infty}2(i+3)\gamma_i,
\end{equation}
\begin{equation*}
J(m)\cap X'_{n(m)}\supset \{x^{n(m)}_{i(1)}<x^{n(m)}_{i(2)}<\cdots<x^{n(m)}_{i(m+2)}<x^{n(m)}_{i(m+3)}\}.
\end{equation*}
We put $\beta^n=0_n$ for each $n(m-1)<n<n(m+1)$; the coefficients $\alpha^{n(m)}_k$, $k=1,\dots,n(m)$ are gained as the unique solution of the linear equation $A_{n(m)}\alpha^{n(m)}=\beta^{n(m)}$, where $\beta^{n(m)}_1=0$, $\beta^{n(m)}_{i(j)}=(-1)^j\gamma_m$, $j=1,\dots,m+3$ and $\beta^{n(m)}_i=0$ otherwise.

Since by Claim~\ref{cl:3}, $\alpha^n_k=0$ for $n\neq n(m)$, $\vert\alpha^{n(m)}_k\vert\le \gamma_m$ and by \eqref{e:10} there are at most $2(m+3)$ nonzero coefficients $\alpha^{n(m)}_k$,
one can see that by our choice of the $\beta$'s
\begin{equation*}
\sum_{n=2}^{\infty}\sum_{k=1}^n\vert\alpha^n_k\vert\le\sum_{m=1}^{\infty}\sum_{k=1}^{n(m)}\vert\alpha^{n(m)}_k\vert\le
\sum_{m=1}^{\infty}2(m+3)\gamma_m<\infty.
\end{equation*}
Thus, the function $f$ given by the above coefficients $\alpha^n_k$ and the formula \eqref{e:14} belongs to the space $\A$.
Using the equality $g_{n(m-1)}(x_0)=x_0$ and \eqref{e:17}, \eqref{e:15} we get for $j\le m+3$ odd
\begin{equation*}\label{e:18}
f(x^{n(m)}_{i(j)})\le x_0+\omega(g_{n(m-1)},J(m))-\gamma_m+\sum_{i=m+1}^{\infty}2(i+3)\gamma_i\le x_0-\varepsilon(m)
\end{equation*}
and analogously for $j\le m+3$ even
\begin{equation*}\label{e:19}
f(x^{n(m)}_{i(j)})\ge x_0-\omega(g_{n(m-1)},J(m))+\gamma_m-\sum_{i=m+1}^{\infty}2(i+3)\gamma_i\ge x_0+\varepsilon(m).
\end{equation*}
At the same time $[x^{n(m)}_{i(j)},x^{n(m)}_{i(j+1)}]\subset J(m)=[x_0-\varepsilon(m),x_0+\varepsilon(m)]$, hence the function $f$ has an $(m+2)$-horseshoe (created by the points $x^{n(m)}_{i(1)},x^{n(m)}_{i(2)},\dots,x^{n(m)}_{i(m+3)}$) on the interval $J(m)$. It means that $\htop(f)\ge \log(m+2)$ and $m$ can be arbitrarily large.
\end{proof}

\begin{theorem}\label{thm:A}
There is a universal Banach space $\A\subset C_b(\R)$ such that $\htop(f) = \infty$ for every non-zero $f$ from $\A$.
\end{theorem}

\begin{proof}
Take $p_n = 2^{-n}$ for $n \geq 0$ and $\{q_n\}_{n \geq 0}$ a decreasing sequence such that $q_0=1$, $q_n \geq p_n$
for all $n$, $q_n / p_n \to \infty$, but $q_n \to 0$.
Choose intervals $I_n = [\frac34p_n, \frac54p_n]$ and $J_n = (\frac23p_n, \frac43p_n) \supset I_n$,
both `centered' at $p_n$. Notice also that the $J_n$'s are adjacent:
$\frac23p_n$ is the common boundary point of $J_n$ and $J_{n+1}$.
Now for a function $f\in C([0,1])$, construct $g := \Psi(f) \in C_b(\R)$ as follows, see Figure~\ref{fig1}:
$$
g(y) = \left\{ \begin{array}{ll}
0 & \text{ if } y = 0;\\[2mm]
q_n  \cdot f(\frac{2y}{p_n} - \frac32) & \text{ if } y \in I_n \text{
  for some } n \geq 0;\\[2mm]
0 & \text{ if } y \in \cup_n \partial J_n;\\[2mm]
0 & \text{ if } y \geq\frac43;\\[2mm]
\text{by linear interpolation} & \text{ if } y \in \cup_n (J_n \setminus I_n);\\[2mm]
g(-y) & \text{ if } y < 0;\\[2mm]
\end{array} \right.
$$
Let $\A = \Psi(C([0,1]))\subset C([-\frac43,\frac43])$ equipped with the norm ($q_0=1$)
$$\sup_{y \in \R} |g(y)|=\| g \| = \sup_{y \in I_0} |g(y)|=\|f\|,$$
so $\Psi$ is an isometrical isomorphism and $\A$ is a separable Banach space.

If $f$ is not constant zero, then $g = \Psi(f)$ is not constant zero either and
$$\sup_{y \in I_n} |g(y)| =q_n\|f\|> 0.$$
Fix $d \in \N$ arbitrary. Since $q_n/p_n =q_n/2^{-n}\to \infty$, there is an $n \in \N$ such that
$$q_n\|f\| > 2^{-n+d}=p_{n-d}.$$
Since $\{q_i\}_i$ is decreasing and $g(\pm \partial J_i) = 0$ for all $i$ (where $-J_i = \{ y : -y \in J_i\}$),
it follows that $g(I_i) = g(-I_i) \supset [0, \max J_{n-d+1}]$ or $[-\max J_{n-d+1}, 0]$ for all
$n-d+1 \leq i \leq n$.
Hence, within the intervals $J_{n-d+1}, \dots, J_n$, or within $-J_{n-d+1}, \dots, -J_n$,
we can choose $d$ intervals that form a $d$-horseshoe. This implies that $\htop(g) \geq \log d$. As $d$ was arbitrary, $\htop(g) = \infty$.

For a real, separable Banach space $\B$ we will find an isometrical isomorphism $\Phi:\B \to \A$. Since by the Banach--Mazur Theorem the space $C([0,1])$ is universal, there is an isometrical isomorphism $\tilde\Phi:\B \to C([0,1])$. Using the above constructed isometrical isomorphism $\Psi:C([0,1]) \to \A$, the required $\Phi$ is just $\Psi \circ \tilde \Phi$.
\end{proof}
\vspace{2pt}
\begin{figure}
\begin{center}\epsfig{file=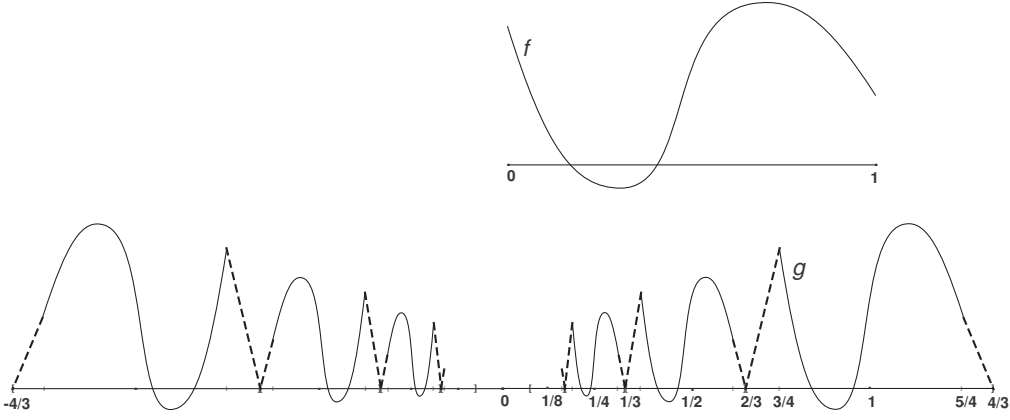,width=14cm}\end{center}
\caption{The maps $f\in C([0,1])$ and $\Psi(f)=g\in C([-\frac43,\frac43])$, $p_n=(\frac{1}{2})^n$, $q_n=(\frac{2}{3})^n$, $n\ge 0$.}\label{fig1}
\end{figure}
\vskip3mm
\begin{remark}
Recall that $f\in C^{\alpha}(\R)$ ($f$ is {\em $\alpha$-H\"older on $\R$}) for some $\alpha \in (0,1)$
if
$$
\sup\left\{\frac{|f(x)-f(y)|}{|x-y|^\alpha}\colon~x,y\in \R, ~0<\vert x-y\vert\le 1\right\}< \infty.
$$
For some fixed $\alpha\in (0,1)$, if we choose $q_n = p_n^\alpha$ and $f \in C^\alpha([0,1])$, then
$\Psi(f)$ is $\alpha$-H\"older on $\R$.
Therefore $\A^\alpha := \Psi(C^\alpha([0,1])) \subset C^\alpha_b(\R)$ is a normed (infinite dimensional) linear space such that $\htop(f) = \infty$ for every non-zero $f$ from $\A^\alpha$.
\end{remark}

\section{Entropy of one-dimensional Banach spaces}

Even if $\dim(\B) = 1$, it is still possible that
$\htop^+(\B) = \infty$. As the following example shows, the upper bound for the entropy need not be attained.

\begin{example}
Let $\B$ be spanned by $f(x) = \sin x$, then $\lambda f$ admits
a $d$-horseshoe whenever $|\lambda| \geq 2\pi d$. Therefore $\htop^+(\B) = \infty$.
\end{example}

The above example also shows that there is no sensible upper bound for $\htop^+(\B)$ in terms of $\dim(\B)$ only. However, $\htop^-(\B)=0$ - see Definition~\ref{def:B}.

In this section we will be investigating the equality $\htop^-(\B)=\htop^+(\B)$ for one-dimensional subspaces $\B$ of $C_b(\R)$: so far we know that for some $\B$,
\begin{itemize}\item $\htop^-(\B)=\htop^+(\B)=\infty$ (easy consequence of Theorem~\ref{thm:A})
\item $\htop^-(\B)=\htop^+(\B)=0$ ($\B$ is spanned by a monotone map)
\end{itemize}

The following statement shows that the entropy can behave extremely rigidly on a one-dimensional subspace of $C_b(\R)$.

\begin{theorem}\label{thm:D}
For any $t \in [0, \infty]$, there exists a function $f \in C_b(\R)$
such that for $\B = \{\lambda f\}_{\lambda\in\R}$ satisfies $\htop^-(\B) = \htop^+(\B)=t$.
\end{theorem}
\begin{proof}
The case $t = 0$ and $t = \infty$ were covered previously,
so let $t \in (0, \infty)$ arbitrary and take an odd integer $d > e^t$.

Let $\theta_a:[0, \infty) \to [0, \infty)$ be a one-parameter family (with
$a \in [0,1]$) of at most $d$-modal continuous  maps
such that for each $a \in [0,1]$, $\theta_a([9,10])\subset [9,10]$ and $\theta_a(x) = x$ whenever $x \notin (9,10)$,
 $\theta_0$ is the identity, and $\theta_1$ has a full
$d$-horseshoe on $[9,10]$.
In the $C^1$ topology for
maps of fixed modality, topological entropy depends continuously on the map,
see \cite[Cor.\ 4.5.5]{alm00}, so there is no loss in generality 
in assuming that $\htop(\lambda \cdot \theta_a)$
is continuous in both $a \in [0,1]$ and $\lambda \in [\frac{9}{10}, \frac{10}{9}]$. (Note that $\htop(\lambda \cdot \theta_a) \equiv 0$ for $\lambda \ge 0$ 
outside this interval.)
Therefore $r_a = \sup_{\lambda \geq 0} \htop(\lambda \cdot \theta_a)$ is 
is continuous in $a$ as well, and $r_0 = 0$, $r_1 = \log d > t$.
Therefore there is $a^*$ such that $r_{a^*} = t$.
Fix $\Theta = \theta_{a^*}$.

Next let $\{ \lambda_i \}_{i \geq 0}$ be a denumeration of the positive
rationals such that $\lambda _1 = 1$ and
\begin{equation}\label{eq:lambdas}
\lambda_{n+1}\le 2\lambda_n \quad \text{ for all } n \geq 0.
\end{equation}
Let $x_n = 4^{-n}$ and $I_n = [0.9 x_n, x_n]$ for $n \geq 0$. Now we set
$$
f(x) = \left\{ \begin{array}{ll}
\lambda_n \cdot \frac{x_n}{10} \cdot \Theta(\frac{10}{x_n} \cdot x) & \text{ if } x \in I_n;\\[2mm]
0 & \text{ if } x = 0;\\[2mm]
10 & \text{ if } x \geq 10;\\[2mm]
\text{by linear interpolation} & \text{ if } x \in (0,10) \setminus \cup_n I_n;\\[2mm]
f(-x) & \text{ if } x < 0.\\[2mm]
\end{array} \right.
$$
Fix $\lambda > 0$.
By assumption \eqref{eq:lambdas} we have that $\lambda f(x) \leq \lambda f(y)$
for all $x \in I_{n+1}$, $y \in I_n$ and $n \geq 0$.
It is not hard to see that every orbit with respect to $\lambda f$ can visit only finitely many
intervals $I_n$, and at most one of them infinitely often.
Therefore, if we choose some $x > 0$, then
$\omega(x)$ can only belong to a single $I_n$, and only if the
diagonal intersects the box $I_n \times \lambda f(I_n)$.
By our choice of $a^*$ (and hence $\Theta$), $\htop(\lambda f|_{I_n}) \leq t$.
Since $x \geq 0$ is arbitrary, $\htop(\lambda f) \leq t$.

For $\varepsilon>0$ let $\lambda^*$ satisfy $ \htop(\lambda^* \Theta)\ge t-\varepsilon$. Since $\{\lambda_n\}_{n \geq 0}$ is dense in $[0, \infty)$ there is some interval $I_m$ such that $\lambda_m\lambda$ is sufficiently close to $\lambda^*$ hence $\htop(\lambda_m\lambda \Theta)\ge t-2\varepsilon$ and also $\htop(\lambda f\vert_{I_m})\ge t-2\varepsilon$. This shows that $\htop(\lambda f) \geq t$, and so we have
$\htop(\lambda f) = t$.

Finally, the dynamics of $-\lambda f$ on $(-\infty, 0]$ is conjugate to the
dynamics of $\lambda f$ on $[0,\infty)$, so also
$\htop(-\lambda f) = t$.
\end{proof}

\end{document}